\documentclass[11pt]{article}
\usepackage{amssymb}
\usepackage[perpage,symbol]{footmisc}
\usepackage{listings}
\usepackage{amsfonts}
\usepackage{enumerate}
\usepackage{amsmath}
\usepackage{cite}
\usepackage{array}
\usepackage{booktabs}
\newcommand\relphantom[1]{\mathrel{\phantom{#1}}}

\begin{document}

\newtheorem{theorem}{Theorem}[section]
\newtheorem{corollary}[theorem]{Corollary}
\newtheorem{definition}[theorem]{Definition}
\newtheorem{proposition}[theorem]{Proposition}
\newtheorem{lemma}[theorem]{Lemma}
\newtheorem{example}[theorem]{Example}
\newenvironment{proof}{\noindent {\bf Proof.}}{\rule{3mm}{3mm}\par\medskip}
\newcommand{\remark}{\medskip\par\noindent {\bf Remark.~~}}

\title{A Class of Six-weight Cyclic Codes and Their Weight Distribution\footnote{This work is supported by the National Natural Science Foundation of China (No. 11071160).}}
\author{Yan Liu\footnote{Corresponding author, Dept. of Math., SJTU, Shanghai, 200240,  liuyan0916@sjtu.edu.cn.}, Haode Yan\footnote{Dept. of Math., Shanghai Jiaotong Univ., Shanghai, 200240, hdyan@sjtu.edu.cn.}, Chunlei Liu\footnote{Dept. of Math., Shanghai Jiaotong Univ., Shanghai,
200240, clliu@sjtu.edu.cn.}}

\date{}
\maketitle
\thispagestyle{empty}

\abstract{In this paper, a family of six-weight cyclic codes over $\mathbb{F}_{p}$  whose duals have three zeros is presented, where  $p$ is  an odd prime. And the weight distribution of these cyclic codes is determined. }



\section{\small{INTRODUCTION}}
Throughout this paper, let $m \geq 3$ be an odd integer and $k$ be a positive integer  such that  gcd$(m,k)=1$. Let $p$ be  an odd prime and $\pi$ be a primitive element of the finite field $\mathbb{F}_{p^{m}}$.

An $[n,l,d]$ linear code $\mathcal{C}$ over $\mathbb{F}_{p}$ is an $l$-dimensional subspace of $\mathbb{F}_{p}^{n}$ with minimum distance $d$. Let $A_{i}$ denote the number of codewords with Hamming weight $i$ in $\mathcal{C}$ of length $n$. The   weight enumerator of $\mathcal{C}$ is defined by $1+A_{1}Z+A_{2}Z^{2}+\cdots+A_{n}Z^{n}$. The sequence $(1, A_{1}, A_{2},\ldots, A_{n})$ is called the weight distribution of the code $\mathcal{C}$. And $\mathcal{C}$ is called cyclic if $(c_{0}, c_{1},  \ldots, c_{n-1}) \in \mathcal{C}$ implies $(c_{n-1}, c_{0},  \ldots, c_{n-2}) \in \mathcal{C}$. By identifying any vector $(c_{0}, c_{1},  \ldots, c_{n-1}) \in \mathbb{F}_{p}^{n}$  with $c_{0}+ c_{1}x+\cdots+c_{n-1}x^{n-1} \in \mathbb{F}_{p}[x]/(x^{n}-1)$, any cyclic code corresponds to an ideal of the polynomial residue class ring $\mathbb{F}_{p}[x]/(x^{n}-1)$. Since $\mathbb{F}_{p}[x]/(x^{n}-1)$ is a principal ideal ring, every cyclic code corresponds to a principal ideal $(g(x))$ of  the multiples of a polynomial $g(x)$ which is the monic polynomial of lowest degree in the ideal. This polynomial  $g(x)$ is called the generator polynomial, and $h(x)=(x^{n}-1)/g(x)$ is referred to as the parity-check polynomial of the code $\mathcal{C}$. A cyclic code is called irreducible if its parity-check polynomial  is irreducible over $\mathbb{F}_{p}$ and reducible, otherwise.

Clearly, the weight distribution gives the minimum distance of the code, and thus the error capability. In addition, the weight distribution of a code allows the computation of the error probability  of error detection and correction with respect to some error detection and error correction algorithms. Thus the study of the weight distribution of a linear code is important in both theory and applications. For cyclic codes, the error correcting capability may not be as good as some other linear codes in general. However, cyclic codes have wide applications in storage and communication systems because they have efficient encoding and decoding algorithms. So the weight distributions of cyclic codes have been interesting subjects of study for many years and are very hard problem in general.

For information on the weight distribution of irreducible cyclic codes, the reader is referred to \cite{1,2,5,6}. Information on the weight distributions of reducible cyclic codes could be found in \cite{7,8,9,10,18,12,13,14,15,19,20}. For the duals of the known cyclic codes whose weight distributions were determined, most of them have at most two zeros, only a few of them have three or more zeros.

The objective of this paper is to determine the weight distribution of a class of six-weight cyclic codes whose duals have three zeros.

This paper is  organized as follows. Section 2 presents some necessary results on quadratic forms which will be needed. Section 3 defines the family of cyclic codes and determines their weight distributions.

\section{\small{QUADRATIC FORMS OVER FINITE FIELDS}}
In this section, we give a brief introduction to the theory of quadratic forms over finite fields which will be needed to calculate the weight distribution of the cyclic codes in the next section.  Quadratic forms have been well studied $($see \cite{11} and the references therein$)$, and have application in design and coding theory.
\begin{definition}
Let $x=\sum_{i=1}^{m}x_{i}\varepsilon_{i}$ where $x_{i}\in \mathbb{F}_{p}$ and $\{\varepsilon_{1}, \varepsilon_{1}, \ldots, \varepsilon_{m}\}$ is a basis for $\mathbb{F}_{p}^{m}$ over $\mathbb{F}_{p}$. The a function $Q(x)$ from $\mathbb{F}_{p}^{m}$ to $\mathbb{F}_{p}$ is a quadratic form over $\mathbb{F}_{p}$ if it can be represented as
\[Q(x)=Q\big(\sum_{i=1}^{m}x_{i}\varepsilon_{i}\big)=\sum_{1\leq i\leq j\leq m}a_{ij}x_{i}x_{j},
\]
where $a_{ij} \in \mathbb{F}_{p}$.
\end{definition}
The rank of the quadratic form $Q(x)$ is defined as the codimension of the $\mathbb{F}_{p}$-vector space $V=\{x \in \mathbb{F}_{p}^{m}|Q(x+z)-Q(x)-Q(z)=0~ for~ all ~x \in \mathbb{F}_{p}^{m} \}$.

For a quadratic form  $F(x)$, there exists a symmetric matrix $A$ of order $m$ over $\mathbb{F}_{p}$ such that $F(x)= XAX'$, where $X=(x_{1}, x_{2},\ldots,x_{m})\in \mathbb{F}_{p}^{m}$ and $X'$ denotes the transpose of $X$. Then there exists a nonsingular matrix $H$ of order $m$ over $\mathbb{F}_{p}$ such that $MAM'$ is a diagonal matrix (\cite{11}). Under the nonsingular linear substitution $X=ZH$ with $Z=(z_{1}, z_{2},\ldots,z_{m})\in \mathbb{F}_{p}^{m}$, then $F(x)=ZMAM'Z'=\sum_{i=1}^{r}d_{i}z_{i}^{2}$, where $r$ is the rank of $F(x)$ and $d_{i} \in \mathbb{F}_{p}^{\ast}$. Let $\triangle= d_{1}d_{2}\cdots d_{r}$ (we assume $\triangle=0$ when $r=0$).  Then the Legendre symbol $(\frac{\triangle}{p})$ is an invariant of $A$ under the action of $H \in GL_{m}(\mathbb{F}_{p})$. The following results is useful in the next section.
\begin{lemma}[\cite{11}]
With the notations as above, we have
\[
\sum_{x \in \mathbb{F}_{p^{m}}}\zeta_{p}^{F(x)}=\begin{cases}
           (\frac{\triangle}{p})p^{m-\frac{r}{2}},&  p \equiv 1\pmod4,\\
            (\frac{\triangle}{p})(\sqrt{-1})^{r}p^{m-\frac{r}{2}},& p \equiv 3\pmod4,
           \end{cases}
\]
for any quadratic form $F(x)$ in $m$ variables of rank $r$ over $\mathbb{F}_{p}$, where $\zeta_{p}$ is a primitive $p$-th root of unity.
\end{lemma}
\begin{lemma}\label{Le:2.3}
Let $F(x)$ be a quadratic form in $m$ variables of rank $r$ over $\mathbb{F}_{p}$, then
 \[
  \sum_{y \in \mathbb{F}_{p}^{\ast}}\sum_{x \in \mathbb{F}_{p^{m}}}\zeta_{p}^{yF(x)}=\begin{cases}
           \pm (p-1)p^{m-\frac{r}{2}},& \textrm{r even},\\
            0,& otherwise.
           \end{cases}
\]
\end{lemma}
The proof is similar to the proof of Lemma 2.2 in \cite{19}, so we omit the details.

For any fixed $(u,v) \in \mathbb{F}_{p^m}^{2}$,  $Q_{u,v}(x)=Tr(ux^{2}+vx^{p^{k}+1})$, where $Tr$ is the trace map from $\mathbb{F}_{p^{m}}$ to $\mathbb{F}_{p}$. Moreover, we have the following result.
\begin{lemma}[\cite{9}]
For any $(u,v) \in \mathbb{F}_{p^{m}}^{2}\backslash \{(0,0)\}$, $Q_{u,v}(x)$ is a quadratic form over $\mathbb{F}_{p}$ with rank $m$, $m-1$, $m-2$.
\end{lemma}

\section{\small{THE CLASS OF SIX-WEIGHT CYCLIC CODES AND THEIR WEIGHT DISTRIBUTION}}
We follow the notations fixed in Section 1. In this section, we first introduce the family of cyclic codes to be studied. Let $h_{0}(x)$, $h_{1}(x)$ and $h_{2}(x)$ be the minimal polynomials of $\pi^{-1}$, $(-\pi)^{-1}$ and $\pi^{-(p^{k}+1)/2}$ over $\mathbb{F}_{p}$, respectively. It is easy to check that $h_{0}(x)$, $h_{1}(x)$ and $h_{2}(x)$ are polynomials of degree $m$ and are pairwise distinct. Define $h(x)=h_{0}(x)h_{1}(x)h_{2}(x)$. Then $h(x)$ has degree $3m$ and is a factor of $x^{p^{m}-1}-1$.

Let $\mathcal{C}_{(p,m,k)}$ be the cyclic code with parity-check polynomial $h(x)$. Then $\mathcal{C}_{(p,m,k)}$ has length $p^{m}-1$ and dimension $3m$. Moreover, it can be expressed as \[\mathcal{C}_{(p,m,k)}=\{\mathbf{c}_{(a,b,c)}: a, b,c \in \mathbb{F}_{p^{m}}\},\] where \[\mathbf{c}_{(a,b,c)}=\big(Tr(a\pi^{t}+b(-\pi)^{t}+c\pi^{(p^{k}+1)t/2})\big)_{t=0}^{p^{m}-2}.\]

Let $h'(x)=h_{1}(x)h_{2}(x)$ and $\mathcal{C}'_{(p,m,k)}$ be the cyclic code with parity-check polynomial $h'(x)$. Then $\mathcal{C}'_{(p,m,k)}$ is a subcode of $\mathcal{C}_{(p,m,k)}$ with dimension $2m$. Zhengchun Zhou and Cunsheng Ding \cite{19} proved that $\mathcal{C}'_{(p,m,k)}$ have three nonzero weights and determined its weight distribution. In this paper, we will show that $\mathcal{C}_{(p,m,k)}$ have six nonzero weights and determine the weight distribution of this class of cyclic codes  $\mathcal{C}_{(p,m,k)}$.

 From now on, we always assume that $\lambda$ is a fixed nonsquare in $\mathbb{F}_{p}$. Since $m$ is odd, it is also a nonsquare in $\mathbb{F}_{p^{m}}$. Then if $SQ$ denotes  the set of all nonzero square elements of $\mathbb{F}_{p^{m}}$, $\lambda x$ runs through all nonsquares of $\mathbb{F}_{p^{m}}$ as $x$ runs through $SQ$. In addition, we have the following result.
\begin{lemma}[\cite{19}]\label{Le:3.1}
$\lambda^{(1+p^{k})/2}=\lambda$ if $k$ is even, and $\lambda^{(1+p^{k})/2}=-\lambda$ otherwise.
\end{lemma}
In terms of exponential sums, the weight of the codeword $\mathbf{c}_{(a,b,c)}=(c_{0}, c_{1},\ldots,c_{p^{m}-2})$ in $\mathcal{C}_{(p,m,k)}$ is given by

\begin{equation*}
 \begin{split}
 W(\mathbf{c}_{(a,b,c)})
   &= \#\{0\leq t\leq p^{m}-2: c_{t}\neq 0\}\\
   &= p^{m}-1-\frac{1}{p}\sum_{t=0}^{p^{m}-2}\sum_{y \in \mathbb{F}_{p}}\zeta_{p}^{yc(t)} \\
   &=  p^{m}-1-\frac{1}{p}\sum_{t=0}^{p^{m}-2}\sum_{y \in \mathbb{F}_{p}}\zeta_{p}^{yTr(a\pi^{t}+b(-\pi)^{t}+c\pi^{(p^{k}+1)t/2}}) \\
   &=   p^{m}-1- \frac{1}{p}\sum_{y \in \mathbb{F}_{p}}\sum_{t=0}^{(p^{m}-3)/2}\big(\zeta_{p}^{yTr((a+b)\pi^{2t+1}+c\pi^{(p^{k}+1)t})} +\zeta_{p}^{yTr((a-b)\pi^{2t}+c\pi^{\frac{p^{k}+1}{2}(2t+1)})}\big) \\
   &=   p^{m}-1- \frac{1}{p}\sum_{y \in \mathbb{F}_{p}}\sum_{x \in SQ}\big(\zeta_{p}^{yTr((a+b)x+cx^{(p^{k}+1)/2})} +\zeta_{p}^{yTr((a-b)\pi x+c(\pi x)^{\frac{p^{k}+1}{2}})}\big) \\
   &=   p^{m}-1- \frac{1}{p}\sum_{y \in \mathbb{F}_{p}}\sum_{x \in SQ}\big(\zeta_{p}^{yTr((a+b)x+cx^{(p^{k}+1)/2})} +\zeta_{p}^{yTr((a-b)\lambda x+c(\lambda x)^{\frac{p^{k}+1}{2}})}\big) \\
   &=   p^{m}-1- \frac{1}{2p}\sum_{y \in \mathbb{F}_{p}}\sum_{x \in \mathbb{F}_{p^{m}}^{\ast}}\big(\zeta_{p}^{yTr((a+b)x^{2}+cx^{p^{k}+1})} +\zeta_{p}^{yTr((a-b)\lambda x^{2}+c\lambda^{\frac{p^{k}+1}{2}} x^{p^{k}+1})}\big) \\
   &=   p^{m}-p^{m-1}- \frac{1}{2p}\sum_{y \in \mathbb{F}_{p}^{\ast}}\sum_{x \in \mathbb{F}_{p^{m}}}\big(\zeta_{p}^{yTr((a+b)x^{2}+cx^{p^{k}+1})} +\zeta_{p}^{yTr((a-b)\lambda x^{2}+c\lambda^{\frac{p^{k}+1}{2}} x^{p^{k}+1})}\big).
 \end{split}
\end{equation*}
It then follows from Lemma \ref{Le:3.1} that $W(\mathbf{c}_{(a,b,c)})=p^{m}-p^{m-1}- \frac{1}{2p}S(a,b,c)$ when $k$ is even, where
\begin{equation}\label{Eq:3.1}
S(a,b,c)=\sum_{y \in \mathbb{F}_{p}^{\ast}}\sum_{x \in \mathbb{F}_{p^{m}}}\big(\zeta_{p}^{yTr((a+b)x^{2}+cx^{p^{k}+1})} +\zeta_{p}^{yTr((a-b)\lambda x^{2}+c\lambda x^{p^{k}+1}}\big),
 \end{equation}
 and $W(\mathbf{c}_{(a,b,c)})=p^{m}-p^{m-1}- \frac{1}{2p}T(a,b,c)$ when $k$ is odd, where
 \begin{equation}\label{Eq:3.2}
 T(a,b,c)=\sum_{y \in \mathbb{F}_{p}^{\ast}}\sum_{x \in \mathbb{F}_{p^{m}}}\big(\zeta_{p}^{yTr((a+b)x^{2}+cx^{p^{k}+1})} +\zeta_{p}^{yTr((a-b)\lambda x^{2}-c\lambda x^{p^{k}+1}}\big).
 \end{equation}
Based on the discussions above, the weight distribution of the code $\mathcal{C}_{(p,m,k)}$ is completely determined by the value distribution of $S(a,b,c)$ and $T(a,b,c)$. Before doing this, we first give a notation.
For any $(u,v) \in \mathbb{F}_{p^m}^{2}$,
\begin{equation}\label{Eq:3.3}
D(u, v)=\sum_{y \in \mathbb{F}_{p}^{\ast}}\sum_{x \in \mathbb{F}_{p^{m}}}\zeta_{p}^{yQ_{u,v}(x)}=\sum_{y \in \mathbb{F}_{p}^{\ast}}\sum_{x \in \mathbb{F}_{p^{m}}}\zeta_{p}^{yTr(ux^{2}+vx^{p^{k}+1})}.
\end{equation}
The following lemmas are very important to establish the value distribution of $S(a,b,c)$ and $T(a,b,c)$.
\begin{lemma}\label{Le:3.3}
Let $D(u,v)$ be defined by (\ref{Eq:3.3}).
\[D(u,0)=
\begin{cases}
(p-1)p^{m}& u=0\\
0& u\neq 0.
\end{cases}
\]
\end{lemma}
\begin{proof}
By Eq. (\ref{Eq:3.3}),
\[
D(u, 0)=\sum_{y \in \mathbb{F}_{p}^{\ast}}\sum_{x \in \mathbb{F}_{p^{m}}}\zeta_{p}^{yQ_{u,0}(x)}=\sum_{y \in \mathbb{F}_{p}^{\ast}}\sum_{x \in \mathbb{F}_{p^{m}}}\zeta_{p}^{yTr(ux^{2})}.
\]
Then $D(0, 0)=(p-1)p^{m}$. If $u\neq 0$, $Q_{u,0}(x)=Tr(ux^{2})$ is a quadratic form of rank $m$ over $\mathbb{F}_{p}$. So $D(u, 0)=0$ by Lemma \ref{Le:2.3}.
\end{proof}
\begin{lemma}\label{Le:3.4}
Let $D(u,v)$ be defined by (\ref{Eq:3.3}). Then for any fixed $v \in \mathbb{F}_{p^{m}}^{\ast}$, as $u$ runs through $\mathbb{F}_{p^{m}}$, the value distribution of $D(u,v)$ is given by Table 1.
\end{lemma}
\begin{table}[htbp]\label{T:1}
\caption{Value distribution of $D(u,v)$ for fixed $v \in \mathbb{F}_{p^{m}}^{\ast}$}
\centering
\begin{tabular}{ll}
 \hline
 Value& Frequency\\
 \hline
  0 & $p^{m}-p^{m-1}$\\
 $(p-1)p^{\frac{m+1}{2}}$ & $\frac{1}{2}(p^{m-1}+p^{\frac{m-1}{2}})$\\
 $-(p-1)p^{\frac{m+1}{2}}$ &  $\frac{1}{2}(p^{m-1}-p^{\frac{m-1}{2}})$\\
 \hline
\end{tabular}
\end{table}
\begin{proof}
As in Eq. (\ref{Eq:3.3}),
\begin{equation*}
D(u, v)=\sum_{y \in \mathbb{F}_{p}^{\ast}}\sum_{x \in \mathbb{F}_{p^{m}}}\zeta_{p}^{yQ_{u,v}(x)}.
\end{equation*} Then for any $v \in \mathbb{F}_{p^{m}}^{\ast}$, by Lemma \ref{Le:2.3},the values of $D(u, v)$ takes on only the values from the set \{0, $\pm(p-1)p^{\frac{m+1}{2}}$\}.
To determine the distribution of $D(u,v)$ for any fixed $v \in \mathbb{F}_{p^{m}}^{\ast}$, we define
\[
n_{\epsilon}=\#\{u \in \mathbb{F}_{p^{m}}: D(u,v)=\epsilon(p-1)p^{\frac{m+1}{2}}\},
\]
where $\epsilon=0, \pm1$.
Then we have
\begin{equation}\label{Eq:3.4}
\sum_{u \in \mathbb{F}_{p^{m}}}D(u,v)=(n_{1}-n_{-1})(p-1)p^{\frac{m+1}{2}}
\end{equation}
and
\begin{equation}\label{Eq:3.5}
\sum_{u \in \mathbb{F}_{p^{m}}}D^{2}(u,v)=(n_{1}+n_{-1})(p-1)^{2}p^{m+1}.
\end{equation}
On the other hand, it follows from (\ref{Eq:3.3}) that
\begin{equation}\label{Eq:3.6}
\begin{split}
\sum_{u \in \mathbb{F}_{p^{m}}}D(u,v)&=\sum_{u \in \mathbb{F}_{p^{m}}}\sum_{y \in \mathbb{F}_{p}^{\ast}}\sum_{x \in \mathbb{F}_{p^{m}}}\zeta_{p}^{yTr(ux^{2}+vx^{p^{k}+1})}\\
&=\sum_{y \in \mathbb{F}_{p}^{\ast}}\sum_{x \in \mathbb{F}_{p^{m}}}\zeta_{p}^{yTr(vx^{p^{k}+1})}\sum_{u \in \mathbb{F}_{p^{m}}}\zeta_{p}^{yTr(ux^{2})}\\
&=(p-1)p^{m}
 \end{split}
\end{equation}
and
\begin{equation}\label{Eq:3.7}
\begin{split}
&\sum_{u \in \mathbb{F}_{p^{m}}}D^{2}(u,v)\\
&=\sum_{u \in \mathbb{F}_{p^{m}}}\big(\sum_{y_{1} \in \mathbb{F}_{p}^{\ast}}\sum_{x_{1} \in \mathbb{F}_{p^{m}}}\zeta_{p}^{y_{1}Tr(ux_{1}^{2}+vx_{1}^{p^{k}+1})}\big)
\big(\sum_{y_{2} \in \mathbb{F}_{p}^{\ast}}\sum_{x_{2} \in \mathbb{F}_{p^{m}}}\zeta_{p}^{y_{2}Tr(ux_{2}^{2}+vx_{2}^{p^{k}+1})}\big)\\
&=\sum_{u \in \mathbb{F}_{p^{m}}}\sum_{(y_{1},y_{2}) \in \mathbb{F}_{p}^{\ast 2}}\sum_{(x_{1},x_{2}) \in \mathbb{F}_{p^{m}}^{2}}\zeta_{p}^{Tr(u(y_{1}x_{1}^{2}+y_{2}x_{2}^{2})
+v(y_{1}x_{1}^{p^{k}+1}+y_{2}x_{2}^{p^{k}+1}))}\\
&=(p-1)p^{m}+\sum_{u \in \mathbb{F}_{p^{m}}}\sum_{(y_{1},y_{2}) \in \mathbb{F}_{p}^{\ast 2}}\sum_{(x_{1},x_{2}) \in \mathbb{F}_{p^{m}}^{\ast2}}\zeta_{p}^{Tr(u(y_{1}x_{1}^{2}+y_{2}x_{2}^{2})
+v(y_{1}x_{1}^{p^{k}+1}+y_{2}x_{2}^{p^{k}+1}))}\\
&=(p-1)p^{m}+\sum_{(y_{1},y_{2}) \in \mathbb{F}_{p}^{\ast 2}}\sum_{(x_{1},x_{2}) \in \mathbb{F}_{p^{m}}^{\ast2}}\zeta_{p}^{Tr(v(y_{1}x_{1}^{p^{k}+1}+y_{2}x_{2}^{p^{k}+1}))}\sum_{u \in \mathbb{F}_{p^{m}}}\zeta_{p}^{Tr(u(y_{1}x_{1}^{2}+y_{2}x_{2}^{2}))}\\
&=(p-1)p^{m}+\sum_{(y_{1},y_{2}) \in \mathbb{F}_{p}^{\ast 2}}\sum_{(x_{1},x_{2}) \in \mathbb{F}_{p^{m}}^{\ast2}}\zeta_{p}^{Tr(v(y_{1}x_{1}^{p^{k}+1}-y_{2}x_{2}^{p^{k}+1}))}\sum_{u \in \mathbb{F}_{p^{m}}}\zeta_{p}^{Tr(u(y_{1}x_{1}^{2}-y_{2}x_{2}^{2}))}\\
&=(p-1)p^{m}+p^{m}\sum_{t^{2} \in Sq}\sum_{\begin{subarray}{l}
y_{2} \in \mathbb{F}_{p}^{\ast }\\
y_{1}=y_{2}t^{2}
\end{subarray}}\sum_{\begin{subarray}{l}
x_{1} \in \mathbb{F}_{p^{m}}^{\ast}\\
x_{2}^{2}=x_{1}^{2}t^{2}
\end{subarray}}\zeta_{p}^{Tr(v(y_{1}x_{1}^{p^{k}+1}
-y_{2}x_{2}^{p^{k}+1}))}\\
&= (p-1)p^{m}+2p^{m}\sum_{t^{2} \in Sq}\sum_{y_{2} \in \mathbb{F}_{p}^{\ast }}\sum_{x_{1} \in \mathbb{F}_{p^{m}}^{\ast}}\zeta_{p}^{Tr(v(y_{2}t^{2}x_{1}^{p^{k}+1}
-y_{2}t^{p^{k}+1}x_{2}^{p^{k}+1}))}\\
&=(p-1)p^{m}+2p^{m}\frac{p-1}{2}(p-1)(p^{m}-1)\\
&=(p-1)p^{2m},
\end{split}
\end{equation}
where in the sixth identity we use $Sq$ to denote the set of square elements in $\mathbb{F}_{p}^{\ast}$ and  in the eighth identity we used the fact that $t^{p^{k}+1}=t$ since $t  \in \mathbb{F}_{p}$. Combining Eqs. (\ref{Eq:3.4})-(\ref{Eq:3.7}), we get
\[
n_{1}=\frac{1}{2}(p^{m-1}+p^{\frac{m-1}{2}}),
\]
\[
n_{-1}=\frac{1}{2}(p^{m-1}-p^{\frac{m-1}{2}}).
\]
Then we have $n_{0}=p^{m}-n_{1}-n_{-1}=p^{m}-p^{m-1}$.
\end{proof}
The value distribution of $S(a,b,c)$ will be determined in the following.
\begin{lemma}\label{Le:3.2}
Let $k$ be even and $S(a,b,c)$ be defined by (\ref{Eq:3.1}), then for any $(a,b,c) \in \mathbb{F}_{p^{m}}^{3}$, $S(a,b,c)$ takes values from the set $\{0, (p-1)p^{m}, 2(p-1)p^{m}, \pm(p-1)p^{\frac{m+1}{2}}, \pm2(p-1)p^{\frac{m+1}{2}}\}$.
\end{lemma}
\begin{proof}
Following the notation above, we have
$S(a,b,c)=D(a+b,c)+D(a-b,c)$.
Case I. In the case of $a=b=c=0$, $D(a+b,c)=D(a-b,c)=(p-1)p^{m}$, so $S(a,b,c)=2(p-1)p^{m}$.

\noindent Case II. In the case of  $c=0, a=-b\neq 0$ or $c=0, a=b\neq 0$,
exactly one of $Q(a+b,c)$ and $Q(a-b,c)$ has rank $m$, the other has rank $0$. Then by Lemma \ref{Le:2.3}, we have $S(a,b,c)=(p-1)p^{m}$.

\noindent Case III. In the case of $c\neq0, a+b\neq 0, a-b\neq 0$, again by Lemma \ref{Le:2.3},  $S(a,b,c)\neq 0$ only if $Q(a+b,c)$ or $Q(a-b,c)$ has even rank. Thus $S(a,b,c)=\pm(p-1)p^{\frac{m+1}{2}}$ if $Q(a+b,c)$ has rank $m$ or $m-2$ and $Q(a-b,c)$ has rank $m-1$ or $Q(a-b,c)$ has rank $m$ or $m-2$ and $Q(a+b,c)$ has rank $m-1$. $S(a,b,c)=\pm2(p-1)p^{\frac{m+1}{2}}$ if $Q(a+b,c)$ has rank $m-1$ and $Q(a-b,c)$ has rank $m-1$. And otherwise $S(a,b,c)=0$. This completes the proof.
\end{proof}

\begin{theorem}\label{Th:3.5}
Let $k$ be even and  $S(a,b,c)$ be defined by (\ref{Eq:3.1}). Then as $(a,b,c)$ runs through $\mathbb{F}_{p^{m}}^{3}$, the value distribution of $S(a,b,c)$ is given by Table 2.
\end{theorem}
\begin{table}[htbp]\label{T:2}
\caption{Value Distribution of $S(a,b,c)$}
\centering
\begin{tabular}{ll}
 \hline
 Value& Frequency\\
 \hline
 $2(p-1)p^{m}$ &1\\
 $(p-1)p^{m}$ &  $2(p^{m}-1)$\\
 $(p-1)p^{\frac{m+1}{2}}$ & $(p^{m}-1)(p^{m}-p^{m-1})(p^{m-1}+p^{\frac{m-1}{2}})$\\
 $-(p-1)p^{\frac{m+1}{2}}$ &  $(p^{m}-1)p^{m}-p^{m-1})(p^{m-1}-p^{\frac{m-1}{2}})$\\
 $2(p-1)p^{\frac{m+1}{2}}$ & $\frac{1}{4}(p^{m}-1)(p^{m-1}+p^{\frac{m-1}{2}})^{2}$\\
 $-2(p-1)p^{\frac{m+1}{2}}$ & $\frac{1}{4}(p^{m}-1)(p^{m-1}-p^{\frac{m-1}{2}})^{2}$\\
 $0$ & $(p^{m}-1)(p^{2m}+\frac{3}{2}p^{2(m-1)}-2p^{2m-1}+p^{m}-\frac{1}{2}p^{m-1}-1)$\\
 \hline
\end{tabular}
\end{table}
\begin{proof}
The distribution of $S(a,b,c)=(p-1)p^{m}$ or $2(p-1)p^{m}$ can be easily obtained by Lemma \ref{Le:3.2}. To determine the distribution of the other values, we define
\[
N_{\epsilon}=\#\{(a,b,c)\in \mathbb{F}_{p^{m}}^{3} : S(a,b,c)=\epsilon (p-1)p^{\frac{m+1}{2}}\}
\]
where $\epsilon=0,\pm1,\pm2$.
Then we have
\[
\begin{split}
N_{1}&=\#\{(a,b,c)\in \mathbb{F}_{p^{m}}^{3} : S(a,b,c)= D(a+b,c)+D(a-b,c)= (p-1)p^{\frac{m+1}{2}}\}\\
&=\#\{(u_{1},u_{2},c)\in \mathbb{F}_{p^{m}}^{3} : D(u_{1},c)+D(u_{2},c)= (p-1)p^{\frac{m+1}{2}}\}\\
&=\#\{(u_{1},u_{2}\in \mathbb{F}_{p^{m}}^{2}, ,c \in \mathbb{F}_{p^{m}}^{\ast}: D(u_{1},c)+D(u_{2},c)= (p-1)p^{\frac{m+1}{2}}\}+\\
 &\relphantom{=} {}\#\{(u_{1},u_{2}\in \mathbb{F}_{p^{m}}^{2} : D(u_{1},0)+D(u_{2},0)= (p-1)p^{\frac{m+1}{2}}\}\\
&=\#\{(u_{1},u_{2}\in \mathbb{F}_{p^{m}}^{2}, ,c \in \mathbb{F}_{p^{m}}^{\ast}: D(u_{1},c)+D(u_{2},c)= (p-1)p^{\frac{m+1}{2}}\}\\
&=\#\{(u_{1},u_{2}\in \mathbb{F}_{p^{m}}^{2}, ,c \in \mathbb{F}_{p^{m}}^{\ast}: D(u_{1},c)=0, D(u_{2},c)= (p-1)p^{\frac{m+1}{2}}\}+\\
&\relphantom{=} {}\#\{(u_{1},u_{2}\in \mathbb{F}_{p^{m}}^{2}, ,c \in \mathbb{F}_{p^{m}}^{\ast}: D(u_{1},c)=(p-1)p^{\frac{m+1}{2}}, D(u_{2},c)=0 \}\\
&=2n_{0}n_{1}(p^{m}-1)\\
&=(p^{m}-1)(p^{m}-p^{m-1})(p^{m-1}+p^{\frac{m-1}{2}}),
\end{split}
\]
where the second part of the third identity is 0 by Lemma \ref{Le:3.3} and the sixth identity is obtained by Lemma \ref{Le:3.4}.

Similarly, we get
\[
N_{-1}=2n_{0}n_{-1}(p^{m}-1)=(p^{m}-1)(p^{m}-p^{m-1})(p^{m-1}-p^{\frac{m-1}{2}}),
\]
\[
N_{2}=n_{1}^{2}(p^{m}-1)=\frac{1}{4}(p^{m}-1)(p^{m-1}+p^{\frac{m-1}{2}})^{2},
\]
\[
N_{-2}=n_{1}^{2}(p^{m}-1)=\frac{1}{4}(p^{m}-1)(p^{m-1}-p^{\frac{m-1}{2}})^{2}
\]
and
\[
\begin{split}
N_{0}&=p^{3m}-1-2(p^{m}-1)-N_{1}-N_{-1}-N_{2}-N_{-2}\\
&=(p^{m}-1)(p^{2m}+\frac{3}{2}p^{2(m-1)}-2p^{2m-1}+p^{m}-\frac{1}{2}p^{m-1}-1)
\end{split}
\]
\end{proof}
\remark Following the notations above, we have
$T(a,b,c)=D(a+b,c)+D(a-b,-c)$. It can be shown that  the value distribution of $T(a,b,c)$ in the case of $k$ is odd is the same as the value distribution of $S(a,b,c)$ in the case of $k$ is even.

The following is the main result of this paper.
\begin{theorem}\label{Th:3.6}
$\mathcal{C}_{(p,m,k)}$ is a  cyclic code over $\mathbb{F}_{p}$ with parameters $[p^{m}-1, 3m, \frac{p-1}{2}p^{m-1}]$. Furthermore, the weight distribution of $\mathcal{C}_{(p,m,k)}$ is given by Table 3.
\end{theorem}
\begin{table}[htbp]\label{T:2}
\caption{Weight Distribution of $\mathcal{C}_{(p,m,k)}$}
\centering
\begin{tabular}{ll}
 \hline
 Weight& Frequency\\
 \hline
  $0$ &1\\
 $\frac{p-1}{2}p^{m-1}$ &  $2(p^{m}-1)$\\
 $\frac{p-1}{2}(2p^{m-1}-p^{\frac{m-1}{2}})$ & $(p^{m}-1)(p^{m}-p^{m-1})(p^{m-1}+p^{\frac{m-1}{2}})$\\
 $\frac{p-1}{2}(2p^{m-1}+p^{\frac{m-1}{2}})$ &  $(p^{m}-1)p^{m}-p^{m-1})(p^{m-1}-p^{\frac{m-1}{2}})$\\
 $(p-1)(p^{m-1}-p^{\frac{m-1}{2}})$ & $\frac{1}{4}(p^{m}-1)(p^{m-1}+p^{\frac{m-1}{2}})^{2}$\\
 $(p-1)(p^{m-1}+p^{\frac{m-1}{2}})$ & $\frac{1}{4}(p^{m}-1)(p^{m-1}-p^{\frac{m-1}{2}})^{2}$\\
 $(p-1)p^{m-1}$ & $(p^{m}-1)(p^{2m}+\frac{3}{2}p^{2(m-1)}-2p^{2m-1}+p^{m}-\frac{1}{2}p^{m-1}-1)$\\
 \hline
\end{tabular}
\end{table}
\begin{proof}
The length and dimension of $\mathcal{C}_{(p,m,k)}$ follow directly from its definition. The minimal weight and weight distribution of $\mathcal{C}_{(p,m,k)}$ follow from Eqs. (\ref{Eq:3.1}) and  (\ref{Eq:3.2}), Theorem \ref{Th:3.5} and the Remark above.
\end{proof}
\begin{example}
Let $p=3$, $m=3$ and $k=1$. The the code $\mathcal{C}_{(3,3,1)}$ is a $[26,9,9]$ cyclic code over $\mathbb{F}_{3}$ with weight enumerator
\[
1+52z^{9}+936z^{12}+5616z^{15}+10036z^{18}+2808z^{21}+234z^{24},
\]
which confirms the weight distribution in Table 3.
\end{example}
\begin{example}
Let $p=3$, $m=5$ and $k=2$. The the code $\mathcal{C}_{(3,5,1)}$ is a $[242,15,81]$ cyclic code over $\mathbb{F}_{3}$ with weight enumerator
\[
1+484z^{81}+490050z^{144}+3828360z^{153}+7193692z^{162}+2822688z^{171}+313632z^{180},
\]
which confirms the weight distribution in Table 3.
\end{example}
\begin{example}
Let $p=3$, $m=7$ and $k=2$. The the code $\mathcal{C}_{(3,7,2)}$ is a $[2186,21,729]$ cyclic code over $\mathbb{F}_{3}$ with weight enumerator
\[
\begin{split}
&1+4372z^{729}+312344424z^{1404}+2409514128z^{1431}+5231766916z^{1458}+2237405976z^{1485}\\
&+269317386z^{1512},
\end{split}
\]
which confirms the weight distribution in Table 3.
\end{example}

\end{document}